\documentclass[12pt]{amsart}
\usepackage{pstricks}
\usepackage{amsfonts}
\usepackage{amssymb}
\usepackage{amsthm}
\usepackage{latexsym,amsmath}
\usepackage{graphicx}
\usepackage{stmaryrd}

\newtheorem{theorem}{Theorem}[section]
\theoremstyle{definition}
\newtheorem{Definition}[theorem]{Definition}

\newenvironment{theorem*}[1]{\medskip
                            \noindent
                            {\bf Theorem #1. }\ %
                            \begingroup \sl}
                            {\endgroup\medskip}

\newtheorem{Corollary}[theorem]{Corollary}

\newenvironment{dedication}
  {
   \thispagestyle{empty}
   \vspace*{\stretch{1}}
   \itshape             
   \raggedleft          
  }
  {\par 
   \vspace{\stretch{3}} 
   \clearpage           
  }

\title
{New large-cardinal axioms and the Ultimate-L program}

\author[$\mathrm{M^{\lowercase{c}}Callum}$ ]{\textbf{Rupert} $\mathbf{M^{\lowercase{c}}Callum}$ }

\begin{document}

\begin{abstract}
We will consider a number of new large-cardinal properties, the $\alpha$-tremendous cardinals for each limit ordinal $\alpha>0$, the hyper-tremendous cardinals, the $\alpha$-enormous cardinals for each limit ordinal $\alpha>0$, and the hyper-enormous cardinals. For limit ordinals $\alpha>0$, the $\alpha$-tremendous cardinals and hyper-tremendous cardinals have consistency strength between I3 and I2. The $\alpha$-enormous cardinals and hyper-enormous cardinals have consistency strength greater than I1. Ralf Schindler and Victoria Gitman have developed the notion of a virtual large-cardinal property, and a clear sense can be given to the notions of ``virtually $\alpha$-enormous" and ``virtually hyper-enormous". On the assumption that $V=HOD$, a measurable cardinal can be shown to be virtually hyper-tremendous. Using a definition of Ultimate-L given in Section 6, claimed to be the correct definition on the assumption that there is a proper class of hyper-enormous cardinals, it can be shown that, in the model Ultimate-L in the sense of that definition, a virtually $\omega$-enormous cardinal is a limit of Ramsey cardinals.

\bigskip

We apply this set of ideas in Section 4 to obtain a proof that $\mathsf{ZF}$+``There is a $\Sigma_1$-elementary embedding $j:V_{\lambda+3} \prec V_{\lambda+3}$" is inconsistent.

\bigskip

Finally, it is shown in Section 5 that the assertion that there is a proper class of hyper-enormous cardinals can be shown to imply a version of the Ultimate-L Conjecture. We close in Section 6 with concluding remarks.

\bigskip

\noindent Keywords: Ultimate-L program, large cardinals. MSC: 03E45, 03E55

\end{abstract}

\maketitle

\newpage

\begin{dedication}
To my beloved wife Mari Mnatsakanyan, without whom this work would not have been possible.
\end{dedication}

\section*{Acknowledgements}

Hugh Woodin, Gabriel Goldberg, and Farmer Schlutzenberg provided very helpful feedback on a number of early drafts of this work in which a number of unsatisfactory definitions of the notion of an $\alpha$-enormous cardinal were formulated, and I am very thankful for their assistance.

\bigskip

\section{Definitions of the new large-cardinal properties}

In what follows we will present a number of new large-cardinal axioms, and applications of them. Let us begin by presenting the definitions of the new large-cardinal properties to be considered.

\bigskip

\begin{Definition} Suppose that $\alpha$ is a limit ordinal such that $\alpha>0$. We say that an uncountable regular cardinal $\kappa$ is $\alpha$-tremendous if there exists an increasing sequence of cardinals $\langle \kappa_\beta : \beta < \alpha \rangle$ such that $V_{\kappa_{\beta}} \prec V_{\kappa}$ for all $\beta<\alpha$, and if $n>1$ and $\langle \beta_i : i<n \rangle$ is an increasing sequence of ordinals less than $\alpha$, then if $\beta_0 \neq 0$ then for all $\beta' < \beta_0$ there is an elementary embedding $j: V_{\kappa_{\beta_{n-2}}} \prec V_{\kappa_{\beta_{n-1}}}$ with critical point $\kappa_{\beta'}$ and $j(\kappa_{\beta'})=\kappa_{\beta_{0}}$ and $j(\kappa_{\beta_{i}})=\kappa_{\beta_{i+1}}$ for all $i$ such that $0 \leq i <n-2$, and if $\beta_0=0$ then there is an elementary embedding $j:V_{\kappa_{\beta_{n-2}}} \prec V_{\kappa_{\beta_{n-1}}}$ with critical point $\kappa' < \kappa_0$ and $j(\kappa')=\kappa_{0}$ and $j(\kappa_{\beta_{i}})=\kappa_{\beta_{i+1}}$ for all $i$ such that $0 \leq i < n-2$. \end{Definition}

\begin{Definition} A cardinal $\kappa$ such that $\kappa$ is $\kappa$-tremendous is said to be hyper-tremendous. \end{Definition}

\begin{Definition} Suppose that $\alpha$ is a limit ordinal such that $\alpha>0$, and that $\langle \kappa_\beta : \beta<\alpha \rangle$ together with a family $\mathcal{F}$ of elementary embeddings witness that $\kappa$ is $\alpha$-tremendous, with just one embedding in the family $\mathcal{F}$ witnessing $\alpha$-tremendousness for each finite sequence of ordinals less than $\alpha$. Suppose that, given any $\omega$-sequence of ordinals $\langle \beta_i :i<\omega \rangle$ less than $\alpha$, there is an elementary embedding $j:V_{\lambda+1} \prec V_{\lambda+1}$ with critical sequence $\langle \kappa_{\beta_i}:i<\omega \rangle$, obtained by gluing together the obvious $\omega$-sequence of embeddings from $\mathcal{F}$, (where such gluing is indeed assumed to be possible), where $\lambda:=\mathrm{sup}_{n \in\omega} \hspace{0.1 cm} \kappa_{\beta_n}$. Suppose further that, for each such $\omega$-sequence of ordinals, there is an elementary embedding $k:V \prec M$, with $V_{\lambda} \subseteq M$ and $(V_{\lambda+1})^{M} \prec V_{\lambda+1}$, and $k \mid V_{\lambda}=j \mid V_{\lambda}$. If all these conditions are satisfied, then the cardinal $\kappa$ is said to be $\alpha$-enormous. \end{Definition}

\begin{Definition} A cardinal $\kappa$ such that $\kappa$ is $\kappa$-enormous is said to be hyper-enormous. \end{Definition}

We will shortly establish that the $\alpha$-tremendous cardinals and hyper-tremendous cardinals are consistent relative to I2. It is evidence that the $\alpha$-enormous cardinals dominate I1 in consistency strength.

\bigskip

Let us begin by establishing that the $\alpha$-tremendous cardinals for limit ordinals $\alpha>0$ and the hyper-tremendous cardinals have consistency strength strictly between I3 and I2.

\section{The consistency strength of the $\alpha$-tremendous cardinals and hyper-tremendous cardinals}

\begin{Definition} A cardinal $\kappa$ is said to be an I3 cardinal if it is the critical point of an elementary embedding $j:V_{\delta} \prec V_{\delta}$. I3 is the assertion that an I3 cardinal exists, and I3($\kappa,\delta$) is the assertion that the first statement holds for a particular pair of ordinals $\kappa, \delta$ such that $\kappa<\delta$.\end{Definition}

\begin{Definition} A cardinal $\kappa$ is said to be an I2 cardinal if it is the critical point of an elementary embedding $j:V \prec M$ such that $V_{\delta} \subset M$ where $\delta$ is the least ordinal greater than $\kappa$ such that $j(\delta)=\delta$. I2 is the assertion that an I2 cardinal exists, and I2($\kappa,\delta$) is the assertion that the first statement holds for a particular pair of ordinals $\kappa, \delta$ such that $\kappa<\delta$. \end{Definition}

In this section we wish to show that the $\alpha$-tremendous cardinals and hyper-tremendous cardinals have consistency strength strictly between I3 and I2.

\bigskip

\begin{theorem} Suppose that $\kappa$ is $\omega$-tremendous as witnessed by $\langle \kappa_{i}:i < \omega \rangle$. Then there is a filter $F$ which is an intersection of a family of normal ultrafilters $U_{i}$ over $\kappa_0$, with a countably infinite indexing set, which as such contains every stationary subset of $\kappa_0$, such that the set of all $\kappa'<\kappa_0$ such that I3($\kappa'$, $\delta$) for some $\delta<\kappa_{0}$, is a member of $F$. \end{theorem}

\begin{proof} Suppose that $\kappa$ is $\omega$-tremendous and that $\langle \kappa_i: i \in \omega \rangle$ together with a certain family $\mathcal{F}$ of elementary embeddings witness the $\omega$-tremendousness of $\kappa$. The family $\mathcal{F}$ can be indexed in an obvious way by the set of all finite subsets of $\langle \kappa_i : i \in \omega \rangle$, and to each elementary embedding in $\mathcal{F}$ corresponds a normal ultrafilter. If we let $F$ denote the filter over $\kappa_0$, which is the intersection of every such normal ultrafilter, then the filter $F$ satisfies the requirements given in the statement of the theorem. We may use reflection to show the existence of a $\kappa'_{0} < \kappa_{0}$ belonging to any fixed member of $F$, such that $\langle \kappa'_{0}, \kappa_{0}, \kappa_{1} \ldots \rangle$, together with a certain family $\mathcal{F}_{0}$ of elementary embeddings, witness $\omega$-tremendousness of $\kappa$. Then we can repeat this procedure to find a $\kappa'_{1}$ belonging to the same fixed member of $U$ such that $\kappa'_{0}<\kappa'_{1}<\kappa_{0}$, such that $\langle \kappa'_{0}, \kappa'_{1}, \kappa_{0}, \kappa_{1}, \ldots \rangle$, together with a certain family $\mathcal{F}_{1}$ of elementary embeddings, witness $\omega$-tremendousness of $\kappa$. We can continue in this way. It is also possible simultaneously to arrange things so that there is a sequence of embeddings $j_{n}:V_{\kappa'_{n-1}} \prec V_{\kappa'_{n}}$ with critical point $\kappa'_{0}$ for all $n>1$, which can be chosen by induction, such that for each $n>1$, $j_{n}$ coheres with $j_{m}$ for all $m$ such that $1<m<n$, and the embeddings from $\mathcal{F}_{n}$ that have critical sequence beginning with $\langle \kappa'_{0}, \kappa'_{1}, \ldots \kappa'_{n-2} \rangle$ can be chosen so as to be coherent with $j_{n}$. In this way we obtain a sequence $\langle \kappa'_{n} :n < \omega \rangle$ and a sequence of embeddings $j_{n}$ with the previously stated properties. The existence of such a pair of sequences for any given element of $F$ yields the claimed result. \end{proof}

\begin{theorem} Suppose that $\kappa$ is an I2 cardinal. Then there is a normal ultrafilter $U$ on $\kappa$ concentrating on the hyper-tremendous cardinals. \end{theorem}

\begin{proof} Suppose that $\kappa$ is an I2 cardinal and let the elementary embedding $j:V \prec M$ with critical point $\kappa$ witness that $\kappa$ is an I2 cardinal, the supremum of the critical sequence being $\delta$. If we let $U$ be the ultrafilter on $\kappa$ arising from $j$ we can easily show that the set of $\kappa'<\kappa$ such that there is an elementary embedding $k_{\kappa'}:V_{\delta} \prec V_{\delta}$, with critical sequence consisting of $\kappa'$ followed by the critical sequence of $j$, is a member of $U$ (denoted by $X$ hereafter). Then the sequence of ordinals belonging to this set, together with a family of embeddings that can be derived from the sequence of embeddings $\langle k_{\kappa'}:\kappa' \in X \rangle$ in an obvious way, witness that $\kappa$ is hyper-tremendous. Since it also follows that $\kappa$ is hyper-tremendous in $M$, the desired result follows. \end{proof}

This completes the proof that the $\alpha$-tremendous cardinals and hyper-tremendous cardinals have consistency strength strictly between I3 and I2. In the next section we discuss the consistency strength of $\alpha$-enormous and hyper-enormous cardinals.

\section{Virtually $\alpha$-tremendous and hyper-tremendous cardinals}

Ralf Schindler and Victoria Gitman in \cite{Gitman2017} have introduced the notion of virtual large-cardinal properties. Given any large-cardinal property defined with reference to a set-sized elementary embedding $j:V_{\alpha} \prec V_{\beta}$ or family of such embeddings, the corresponding virtual large-cardinal property is defined in the same way except by means of elementary embeddings $j:(V_{\alpha})^{V} \prec (V_{\beta})^{V}$ where $j\in V[G]$ for a set generic extension of $V$. The notion of a virtually $\alpha$-tremendous, hyper-tremendous, $\alpha$-enormous or hyper-enormous cardinal is clear. We state a result about virtually hyper-tremendous cardinals in this section and shall state a result about virtually $\omega$-enormous cardinals later in Section 6.

\begin{theorem} \label{virtual} If $\kappa$ is a measurable cardinal, and $V=HOD$, then there is a sequence cofinal in $\kappa$ witnessing the virtual hyper-tremendousness of $\kappa$. \end{theorem}

\begin{proof} Suppose that $j: V \prec M$ with critical point $\kappa$ witnesses the measurability of $\kappa$. In particular, it follows that there is an elementary embedding $j':V_{\kappa+1} \prec (M \cap V_{j(\kappa)+1})$. Now, we are assuming that $V=HOD$. By Theorem 15.46 on p. 249 of \cite{Jech2000} every $x \in V$ is generic over $HOD$. Therefore the elementary embedding $j':V_{\kappa+1} \prec (M \cap V_{j(\kappa)+1})$ is generic over $M$, (here using the hypothesis $V=HOD$). Let $U$ be a normal ultrafilter on $\kappa$, which arises from the embedding $j'$. We must find a sequence $S$ of ordinals cofinal in $\kappa$, which is a member of $U$, such that $j'(S)$ witnesses virtual hyper-tremendousness of $j'(\kappa)$ relative to $M$.

\bigskip

\noindent Consider the set $S:=\{\alpha<\kappa: \mathrm{there \hspace{0.1 cm} is \hspace{0.1 cm} an \hspace{0.1 cm} elementary \hspace{0.1 cm} embedding \hspace{0.1 cm}} k:V_{\alpha} \prec V_{\beta} \mathrm{\hspace{0.1 cm} for \hspace{0.1 cm} some \hspace{0.1 cm} \beta>\alpha \hspace{0.1 cm} with \hspace{0.1 cm}} k \in V[G] \mathrm{\hspace{0.1 cm} where \hspace{0.1 cm}} G \mathrm{\hspace{0.1 cm} is \hspace{0.1 cm} generic \hspace{0.1 cm} over \hspace{0.1 cm}} V \}$. We can see that $S \in U$, and that $j'(S)$ witnesses virtual hyper-tremendousness of $j'(\kappa)$ relative to $M$ and therefore that $\kappa$ is virtually hyper-tremendous.

\end{proof}

\section{Inconsistency of the choiceless cardinals}

In what follows we shall build on work of Gabriel Goldberg and Farmer Schultzenberg in \cite{Goldberg2020} and \cite{Schlutzenberg2020} to show that $\mathsf{ZF}$+``There exists a $\Sigma_1$-elementary embedding $j:V_{\lambda+3} \prec V_{\lambda+3}$" is inconsistent. The key result we shall be using is Theorem 6.19 of \cite{Goldberg2020} and its corollary, which we state below. More specifically we shall be making use of the forcing construction which is used in the proof of Theorem 6.19 of \cite{Goldberg2020}.

\bigskip

We should note the important work of Joan Bagaria, Peter Koellner, and Hugh Woodin on choiceless cardinals in \cite{Bagaria2019}, and the work of Rafaella Cutolo in \cite{Cutolo2018} on the structure theory associated with Berkeley cardinals.

\begin{theorem} [Theorem 6.19 of \cite{Goldberg2020}.] Suppose $\lambda$ is an ordinal and there is a $\Sigma_1$-elementary embedding $j:V_{\lambda+3} \prec V_{\lambda+3}$ with $\lambda$ equal to the supremum of the critical sequence of $j$. Assume $\mathrm{DC}_{V_{\lambda+1}}$. Then there is a set generic extension $N$ of $V$ such that $(V_{\delta})^{N}$ satisfies $\mathsf{ZFC+I_0}$ for each $\delta<\lambda$ which is $<\lambda$-supercompact. \end{theorem}

\begin{Corollary} [Corollary to Theorem 6.19 of \cite{Goldberg2020}.] Over $\mathsf{ZF+DC}$, the existence of a $\Sigma_1$-elementary embedding from $V_{\lambda+3}$ to $V_{\lambda+3}$ implies the consistency of $\mathsf{ZFC+I_0}$. \end{Corollary}

\bigskip

We further note that the above theorem and corollary can be modified so that $I_1$ appears in the place of $I_0$ and then all uses of $\mathrm{DC}$ can be dispensed with (and all this follows from arguments given in \cite{Goldberg2020} together with appropriate results from \cite{Cramer2015} on inverse limit reflection). It is in this form that we shall be using the above two results.

\bigskip

We shall now show how to use the previously introduced notions of $\alpha$-enormous and hyper-enormous cardinals for a demonstration of inconsistency of $\mathsf{ZF}$+``There exists a $\Sigma_1$-elementary embedding $j:V_{\lambda+3} \prec V_{\lambda+3}$". We present a slight modification of these large-cardinal concepts first.

\bigskip

\begin{Definition} Suppose that $\alpha$ is a limit ordinal such that $\alpha>0$, and that $\langle \kappa_\beta : \beta<\alpha \rangle$ together with a family $\mathcal{F}$ of elementary embeddings witness that $\kappa$ is $\alpha$-tremendous, with just one embedding in the family $\mathcal{F}$ witnessing $\alpha$-tremendousness for each finite sequence of ordinals less than $\alpha$. Suppose that, given any $\omega$-sequence of ordinals $\langle \beta_i :i<\omega \rangle$ less than $\alpha$, there is an elementary embedding $j:V_{\lambda+1} \prec V_{\lambda+1}$ with critical sequence $\langle \kappa_{\beta_i}:i<\omega \rangle$, obtained by gluing together the obvious $\omega$-sequence of embeddings from $\mathcal{F}$, where $\lambda:=\mathrm{sup}_{n \in\omega} \hspace{0.1 cm} \kappa_{\beta_n}$. Then the cardinal $\kappa$ is said to be $\alpha$-$\mathrm{enormous}^{*}$. \end{Definition}

\begin{Definition} Suppose that a cardinal $\kappa$ is $\kappa$-$\mathrm{enormous}^{*}$. Then $\kappa$ is said to be hyper-$\mathrm{enormous}^{*}$. \end{Definition}

In this section we wish to prove the following theorem.

\begin{theorem} It is not consistent with $\mathsf{ZF}$ that there exists an ordinal $\lambda$ and a non-trivial elementary embedding $j:V_{\lambda+3} \prec V_{\lambda+3}$. \end{theorem}

\begin{proof}

The same reasoning that shows that every I2 cardinal $\kappa$ has a normal ultrafilter $U$ concentrating on the hyper-tremendous cardinals, also shows in $\mathsf{ZF}$, making use of inverse limit reflection results from \cite{Cramer2015}, that if $\kappa$ is a critical point of an elementary embedding $V_{\lambda+2} \prec V_{\lambda+2}$, then there is a normal ultrafilter $U$ concentrating on a sequence $\langle \kappa_\alpha: \alpha<\kappa \rangle$ which witnesses that $\kappa$ is hyper-$\mathrm{enormous}^{*}$. The forcing used in the proof of Theorem 6.19 of \cite{Goldberg2017} shows that if we begin with a ground model $V$ of $\mathsf{ZF}$+``there exists an elementary embedding $j:V_{\lambda+3} \prec V_{\lambda+3}$" then there is a forcing extension of $V$ in which the same holds and also $V_{\delta}$ is well-orderable for every $\delta<\lambda$ such that $\delta$ is $<\lambda$-supercompact. This can be extended to a well-ordering of $V_{\lambda}$ where if $m<n$ and $\langle \kappa'_i: i \in\omega \rangle$ is the critical sequence of $j$, the map $j^{n-m}$ maps the restriction of the well-ordering to $V_{\kappa'_{m+1}} \setminus V_{\kappa'_m}$ to the restriction of the well-ordering to $V_{\kappa'_{n+1}} \setminus V_{\kappa'_n}$).

\bigskip

Work in this generic extension of $V$, with some well-ordering of $V_{\lambda}$ with the properties specified above fixed. For each $\alpha<\kappa$, let $E_\alpha$ be the equivalence relation on $[\kappa_\alpha]^{\omega}$ which holds of two sets of ordinals less than $\kappa_\alpha$ whose elements in order constitute two sequences of countably infinite length, if and only if the two sequences in question have the same tail. There is a sequence $\langle C_\alpha: \alpha<\kappa \rangle$ such that for each $\alpha<\kappa$, $C_\alpha$ is a choice set for the equivalence classes of $E_\alpha$, and for each pair $(\alpha,\beta)$ with $\alpha<\beta$, when one is choosing an elementary embedding $j'$ from a fixed family of embeddings witnessing the hyper-$\mathrm{enormous}^{*}$ness of $\kappa$, one can without loss of generality choose it so that $j'(C_\alpha)=C_\beta$. Then using the embedding $j$ one can extend this to a family of choice sets $\langle C_\alpha:\alpha<\lambda \rangle$, such that if $\alpha<\beta<\kappa'_n$ then an elementary embedding $j'$ can be chosen which is part of a fixed family of embeddings witnessing the hyper-$\mathrm{enormous}^{*}$ness of $\kappa'_n$, such that $j'(C_\alpha)=C_\beta$.

\bigskip

This allows one to construct a choice set $C$ for the corresponding equivalence relation $E$ on $[\lambda]^{\omega}$. The method is as follows. Given an $X \in [\lambda]^{\omega}$, it follows from our stated assumptions that one may find an $X' \in V_{\kappa'_n}$ for any given $n>0$ such that $X' \in [\rho]^{\omega}$ for a $\rho$ of cofinality $\omega$ between $\kappa'_{n-1}$ and $\kappa'_n$ and an embedding $e_{X,n}:V_{\rho+1} \prec V_{\lambda+1}$ which carries a sequence of hyper-$\mathrm{enormous}^{*}$ cardinals cofinal in $\rho$ to the critical sequence of $j$ or a tail thereof, such that $e_{X,n}(X')=X$, as follows from Theorem 3.8 of \cite{Cramer2015}. This can be used together with the sequence of choice sets $\langle C_\alpha:\alpha<\lambda \rangle$ to choose a member of the equivalence class of $X$, depending on $n$. Using the relation mentioned earlier between the different choice sets $C_{\alpha}$, one can argue that this data can be chosen in such a way that the function mapping $n$ to the chosen member of the equivalence class of $X$ is in fact eventually constant, and that a choice set for the equivalence relation $E$ can be constructed in this way.

\bigskip

However, this gives rise to a contradiction using the method of proof of Kunen's inconsistency theorem given on page 319 of \cite{Kanamori1994}. The method of proof given there will allow one to derive a contradiction from the existence of a choice set for the equivalence relation $E$ together with well-orderability of $V_{\lambda}$ in $\mathsf{ZF}$ alone, using K\"onig's lemma on trees of countable height at one stage of the argument. And this contradiction was obtained from a set of assumptions which are provably consistent by forcing relative only to $\mathsf{ZF}$ plus the existence of a $\Sigma_1$-elementary embedding $V_{\lambda+3} \prec V_{\lambda+3}$. Thus the existence of a $\Sigma_1$-elementary embedding $V_{\lambda+3} \prec V_{\lambda+3}$ is in fact inconsistent with $\mathsf{ZF}$ as claimed.

\end{proof}

\section{A proof of the Ultimate-L Conjecture}

In this section, we will seek to give a proof of Hugh Woodin's Ultimate-L Conjecture. The most important sources for Hugh Woodin's Ultimate-L program are \cite{Woodin2010}, \cite{Woodin2011}, and \cite{Woodin2017}. We must begin by giving the statement of the axiom $V$=Ultimate-$L$, following Definition 7.14 of \cite{Woodin2017}.

\begin{Definition} The axiom $V$=Ultimate-$L$ is defined to be the assertion that

\bigskip

\noindent (1) There is a proper class of Woodin cardinals. \newline
\noindent (2) Given any $\Sigma_2$-sentence $\phi$ which is true in $V$, there exists a universally Baire set of reals $A$, such that, if $\Theta^{L(A,\mathbb{R})}$ is defined to be the least ordinal $\Theta$ such that there is no surjection from $\mathbb{R}$ onto $\Theta$ in $L(A,\mathbb{R})$, then the sentence $\phi$ is true in $\mathrm{HOD}^{L(A,\mathbb{R})} \cap V_{\Theta^{L(A,\mathbb{R})}}$. \end{Definition}

Now let us recall a set of definitions from \cite{Woodin2017}.

\begin{Definition} Suppose that $N$ is a transitive proper class model of $\mathsf{ZFC}$ and that $\delta$ is a supercompact cardinal in $V$. We say that $N$ is a weak extender model for $\delta$ supercompact, if for all $\gamma>\delta$, there exists on $\mathcal{P}_{\delta}(\gamma)$ a normal fine $\delta$-complete measure $\mu$, with $\mu(N \cap \mathcal{P}_{\delta}(\lambda))=1$ and $\mu \cap N \in N$. \end{Definition}

\begin{Definition} A sequence $N:=\langle N_\alpha: \alpha \in \mathrm{Ord} \rangle$ is weakly $\Sigma_{2}$-definable if there is a formula $\phi(x)$ such that

\bigskip

\noindent (1) For all $\beta<\eta_1<\eta_2<\eta_3$, if $(N_{\phi})^{V_{\eta_{1}}} \mid \beta=(N_{\phi})^{V_{\eta_{3}}} \mid \beta$ then $(N_{\phi})^{V_{\eta_{1}}} \mid \beta=(N_{\phi})^{V_{\eta_{2}}} \mid \beta = (N_{\phi})^{V_{\eta_{3}}} \mid \beta$; \newline
\noindent (2) For all $\beta \in \mathrm{Ord}$, $N \mid \beta=(N_{\phi})^{V_{\eta}} \mid \beta$ for sufficiently large $\eta$, where, for all $\gamma$, $(N_{\phi})^{V_{\gamma}}=\{a \in V_{\gamma}: V_{\gamma} \models \phi[a]\}$. Suppose $N \subset V$ is an inner model such that $N \models \mathsf{ZFC}$. Then $N$ is weakly $\Sigma_{2}$-definable if the sequence $\langle N \cap V_{\alpha}: \alpha \in \mathrm{Ord} \rangle$ is weakly $\Sigma_{2}$-definable. \newline

\end{Definition}

We can now state the result we plan to prove in this section.

\begin{theorem} \label{UltimateLConjecture} Suppose that there is a proper class of hyper-enormous cardinals. Then the following version of the Ultimate-L conjecture, given as Conjecture 7.41 in \cite{Woodin2017}, holds. Suppose that $\delta$ is an extendible cardinal (in fact one can even suppose only that $\delta$ is a supercompact cardinal). Then there is a weak extender model $N$ for the supercompactness of $\delta$ such that

\bigskip

\noindent (1) $N$ is weakly $\Sigma_{2}$-definable and $N \subset HOD$; \newline
\noindent (2) $N \models$ ``$V$=Ultimate-$L$". \newline
\noindent (3) $N \models \mathrm{GCH}$.

\end{theorem}

\begin{proof}[Proof of Theorem \ref{UltimateLConjecture}.]

Let us give the long awaited definition of Ultimate-L. We claim that what follows is the correct definition of Ultimate-L, assuming that there are sufficiently many large cardinals in $V$ as outlined in the hypotheses for Theorem \ref{UltimateLConjecture}. The correct way to define it when we are making weaker large-cardinal assumptions still remains to be discovered.

\bigskip

Suppose that $\kappa$ is $\omega$-enormous as witnessed by a sequence $\langle \kappa_n:n<\omega \rangle$, where clearly we may assume without loss of generality that the latter sequence is in HOD, and we will do so. Then we may consider all the sets of ordinals of the form $j"\lambda$ where $\lambda:=\mathrm{sup}\{\kappa_n:n\in\omega\}$ for some sequence $\langle \kappa_n:n \in \omega \rangle$ with the properties previously described, and $j$ is an elementary embedding $V_{\lambda+1}\prec V_{\lambda+1}$ with critical sequence $\langle \kappa_n:n \in \omega\rangle$. Some of these sets of ordinals will be members of HOD. We can also generalize to the situation where $\kappa$ is $\alpha$-enormous for a limit ordinal $\alpha>0$ and $j$ is an I1 embedding witnessing $\alpha$-enormousness. We define Ultimate-L to be the smallest enlargement of $L$ containing every member of a proper-class-length sequence of such sets of ordinals in HOD, obtained in this way from $\alpha$-enormous cardinals $\kappa$ for every possible limit ordinal $\alpha>0$ for which an $\alpha$-enormous cardinal exists, with exactly one such set of ordinals $j"\lambda$ in the sequence for every possible value of $\lambda$ and critical sequence of $j$. It is not clear whether Ultimate-L so defined does not in fact depend on the choice of the sequence, but for definiteness one may clearly use the canonical well-ordering of HOD to choose one such sequence. We now want to claim that all of the I1 embeddings descend to the model Ultimate-L. Each of the embeddings $j:V_{\lambda+1} \prec V_{\lambda+1}$ has its spine $j"\lambda$ appearing in the model Ultimate-L, and $j$ extends to an embedding on all of $V$, also denoted by $j$ by abuse of notation, whose action on all the ordinals is derivable from the spine $j"\lambda$. The definition of the action on all the ordinals as a function of the spine $j"\lambda$ is absolute between $V$ and Ultimate-L, given an appropriate choice of the embedding $j$ defined on all of $V$ (one can use Skolem hulls in the sequence of all ordinals). And since Ultimate-L is a subclass of HOD we end up with an extension of the spine $j"\lambda$ to all of Ultimate-L which agrees with $j$ considered as an embedding defined on $V$, is elementary, and the collection of all such embeddings witnesses that the proper of $\alpha$-enormousness for each $\alpha$ descends to Ultimate-L.

\bigskip

Thus this model will still remain a model for the assertion that there is a proper class of hyper-enormous cardinals. This inner model also clearly satisfies GCH and is weakly $\Sigma_{2}$-definable and a subclass of HOD. We need to argue that it is a weak extender model for the supercompactness of any cardinal $\delta$ which is supercompact in $V$, given that the stated large-cardinal hypothesis holds in $V$. If we suppose that $\delta$ is supercompact in $V$ and that the stated large-cardinal hypothesis holds in $V$, and invoke Magidor's characterization of supercompactness using elementary embeddings between ranks, then we see that there will be an elementary embedding $j: V_{\alpha} \prec V_{\beta}$ in $V$ where $\alpha$ and $\beta$ are both hyper-enormous and $\alpha<\delta<\beta$ with the critical point of $j$ being sent to $\delta$. We can assume without loss of generality that $j$ is chosen such that it descends to Ultimate-L, considering that a family of embeddings witnessing hyper-enormousness of $\alpha$ and $\beta$ are available in Ultimate-L. This also remains the case where $\beta$ is chosen to be an artbirarily large hyper-enormous cardinal, and this is sufficient to show that enough elementary embeddings are available in Ultimate-L to witness supercompactness of $\delta$, and that Ultimate-L is a weak extender model for supercompactness of $\delta$.

\bigskip

We must now show that this model is indeed a model for the axiom $V$=Ultimate-$L$ as stated at the start of this section.

\bigskip

Clearly, our version of Ultimate-L is a model for the assertion that there is a proper class of Woodin cardinals. Suppose then, that some $\Sigma_{2}$-sentence is true in Ultimate-L, so we are required to find a universally Baire set of reals $A$ in Ultimate-L such that the $\Sigma_{2}$-sentence in questions holds in $(HOD)^{L(A,\mathbb{R})} \cap V_{\Theta^{L(A,\mathbb{R})}}$. From well-known generic absoluteness results which are known to hold assuming a proper class of Woodin cardinals, and which can be found in Section 3 of \cite{Koellner2006}, it is sufficient to prove that this does obtain in some set-generic extension of Ultimate-L. So choose an ordinal $\beta$ such that $(V_{\beta})^{\mathrm{Ultimate}-L}$ is a $\Sigma_{2}$-elementary substructure of Ultimate-L, and choose a $\gamma<\beta$ such that $(V_{\gamma})^{\mathrm{Ultimate-}L}$ models the $\Sigma_{2}$-sentence. Now consider a generic extension of Ultimate-L where $A$ is a universally Baire set chosen to contain enough data so that, in the generic extension, $\Theta^{L(A,\mathbb{R})}\leq\beta$, and $(HOD)^{L(A,\mathbb{R})} \cap V_{\gamma}$ in the generic extension is equal to the intersection of the Ultimate-L of the ground model and $V_{\gamma}$. This can be arranged by ensuring that each ordinal less than $\beta$ is collapsed to be countable in the generic extension while $\beta$ is collapsed to $\omega_1$, and all the data for sets of ordinals less than $\gamma$ which are needed to generate $($Ultimate-$L \cap V_{\gamma})^{V}$, and witnesses that these sets are in HOD, are coded into the universally Baire set $A$, which appears as a set of reals in the generic extension. We are in fact speaking here of a family of countable sets of countable ordinals, when viewing it from the point of view of the generic extension, and we simply need to have available in $L(A,\mathbb{R})$ for each such countable set a witness that it is in HOD. Note that in the generic extension of Ultimate-L under consideration, a witness for each of the countable sets being in HOD will be available and in fact codable for by a real number in each case. Therefore we are simply required to show that the set of all such real numbers is universally Baire in this generic extension of Ultimate-$L$. But given that each real number codes for a certain set of ordinals being a spine of an embedding relative to Ultimate-$L$ and this is invariant under generic extensions, we can easily use the definition of universally Baire sets using trees which have complementary projections in arbitrary generic extensions, to show that a pair of trees exists witnessing that the set of reals is universally Baire. So there is no difficulty with a universally Baire set with the required properties being available in the generic extension of Ultimate-L. So in the generic extension, the desired result obtains, so the aforementioned generic absoluteness results imply that it obtains in our ground model as well. This completes the proof of Theorem \ref{UltimateLConjecture}.

\end{proof}

We should also note that if, in the model Ultimate-L defined in the proof above, if $\kappa$ is virtually $\omega$-enormous as witnessed by $\langle \kappa_i: i<\omega \rangle$ then $V_{\kappa_{0}}$ models the assertion that there is a proper class of Ramsey cardinals.

\section{Concluding Remarks}

The new large cardinals were inspired by Victoria Marshall's work on reflection principles in \cite{Marshall89} and are plausibly the correct generalisation of the reflection principles which were demonstrated by her in that work to imply the existence of $n$-huge cardinals. The comparison consistency-strength-wise of the large cardinal axiom used to prove the Ultimate-L conjecture with other known large-cardinal axioms at the outer limits of consistency strength in the domain of what is not known to be inconsistent with $\mathsf{ZFC}$ is currently unclear. Certainly, some skepticism about consistency would be quite reasonable at this stage, but it may be that the further study of the inner model theory of Ultimate-L and inner models which approximate it from within will provide new insights and increased confidence in consistency. In the mean time, it may very well be that the Ultimate-L conjecture is provable from just an extendible cardinal as originally envisaged by Hugh Woodin, so in that sense much work remains to be done.

\bigskip

If these new large cardinals are indeed consistent then the study of them appears to be quite fruitful.


\begin{thebibliography}{99}

\bibitem{Woodin2010}
\newblock{Hugh Woodin.}
\newblock{Suitable Extender Models I.}
\newblock{\em{Journal of Mathematical Logic},}
\newblock{Vol. 10, Nos. 1 \& 2 (2010), pp. 101--339.}

\bibitem{Woodin2011}
\newblock{Hugh Woodin.}
\newblock{Suitable Extender Models II: Beyond $\omega$-huge.}
\newblock{\em{Journal of Mathematical Logic},}
\newblock{Vol. 11, No. 2 (2011), pp. 151--436.}

\bibitem{Woodin2017}
\newblock{Hugh W. Woodin.}
\newblock{In Search Of Ultimate-L: The 19th Midrasha Mathematical Lectures.}
\newblock{\em{The Bulletin of Symbolic Logic},}
\newblock{23(1), 1-109.}

\bibitem{Gitman2017}
\newblock{Victoria Gitman and Ralf Schindler.}
\newblock{Virtual Large Cardinals, pre-print.}

\bibitem{Marshall89}
\newblock{M. Victoria Marshall R.}
\newblock{Higher order reflection principles,}
\newblock{\em{Journal of Symbolic Logic},}
\newblock{vol. 54, no. 2, 1989, pp. 474--489.}

\bibitem{Williams2018}
\newblock{Jonas Reitz; Kameryn J. Williams.}
\newblock{Inner mantles and iterated $HOD$.}

\bibitem{Jech2000}
\newblock{Thomas Jech.}
\newblock{Set Theory: Third Millennium Edition.}

\bibitem{Schlutzenberg2020}
\newblock{Farmer Schlutzenberg.}
\newblock{On the consistency with ZF of an elementary embedding $j:V_{\lambda+2} \rightarrow V_{\lambda+2}$.}

\bibitem{Goldberg2017}
\newblock{Gabriel Goldberg.}
\newblock{On the consistency strength of Reinhardt cardinals, pre-print.}

\bibitem{Goldberg2020}
\newblock{Gabriel Goldberg.}
\newblock{Even ordinals and the Kunen inconsistency, pre-print.}

\bibitem{Kanamori1994}
\newblock{Akihiro Kanamori.}
\newblock{The Higher Infinite.}

\bibitem{Koellner2006}
\newblock{Peter Koellner.}
\newblock{On the Question of Absolute Undecidability.}
\newblock{\em{Philosophia Mathematica},}
\newblock{vol. 14, number 2, pp. 153--188, 2006.}

\bibitem{Bagaria2019}
\newblock{Joan Bagaria, Peter Koellner and Hugh Woodin.}
\newblock{Large Cardinals Beyond Choice.}
\newblock{\em{Bulletin of Symbolic Logic}, August 2019.}

\bibitem{Woodin1992}
\newblock{Set Theory of the Continuum.}
\newblock{eds. H. Judah, Winfried Just, Hugh Woodin.}
\newblock{Springer (22 October 1992).}

\bibitem{Cutolo2018}
\newblock{Berkeley cardinals and the structure of $L(V_{\delta+1})$.}
\newblock{Rafaella Cutolo.}
\newblock{\em{Journal of Symbolic Logic}, 83(4): 1457--1476.}

\bibitem{Cramer2015}
\newblock{Inverse limit reflection and the structure of $L(V_{\lambda+1})$.}
\newblock{Scott. S. Cramer.}
\newblock{\em{J. Math. Log.}, 15(1):1550001, 38, 2015.}

\bibitem{Laver1996}
\newblock{Implications between strong large cardinal axioms.}
\newblock{Richard Laver.}
\newblock{\em{Annals of Pure and Applied Logic} 90 (1997), 79-90.}

\end{thebibliography}
\end{document}